\documentclass[12pt]{article}
\usepackage{amsmath, amssymb, amsfonts, amsthm}
\usepackage{enumitem}
\usepackage{multicol}
\usepackage{hyperref}
\usepackage{tikz}
\usepackage{array}
\usepackage{tabularx}
\tikzstyle{vertex}=[circle, draw, inner sep=0pt, minimum size=6pt]

\allowdisplaybreaks
\def \N {{\mathbb{N}}}
\def \Z {{\mathbb{Z}}}

\def \W {{\mathcal{W}}}

\newtheorem*{theorem*}{Theorem}
\newtheorem{theorem}{Theorem}
\newtheorem{cor}[theorem]{Corollary}

\newtheorem{lemma}[theorem]{Lemma}

\newtheorem{definition}[theorem]{Definition}

\newtheorem{rem}[theorem]{Remark}
\newtheorem*{ex*}{Example}

\newtheorem{pro}[theorem]{Proposition}
\title{Achiral words}
\author {
Shrinit Singh and A. Satyanarayana Reddy\\
Department of Mathematics,\\ Shiv Nadar Institution of Eminence, India-201314\\ (e-mail:
ss101@snu.edu.in, satya.a@snu.edu.in).
}
\date{}
\begin{document}
\maketitle
\begin{abstract}
 A word $w$ in a free group is  {\em achiral} if for every group $G,$ $G_w=G_{w^{-1}},$ where $G_w$ is the image of the word map $w$ on $G.$ We will give few classes of examples of  achiral words. Cocke and Ho asked  whether Engel words are achiral or not. We will prove that it is enough to apply Whitehead's algorithm to check the same.
\end{abstract}
{\bf{AMS(2020)}}:20F10 \\

\section{Introduction}
Let $F_n$ be a free group on $n$  generators $\{x_1,x_2,\ldots,x_n\}.$ A non-empty reduced word $w \in F_n$ is an expression $\prod_{j=1}^s x_{i_j}^{a_j}$, where $i_j \in \{ 1,2, \ldots, n\}, a_j \in \Z\setminus \{0\}$ and for every $j \in \{1,2, \ldots, s\},$ $x_{i_j} \neq (x_{i_{j+1}})^{-1}.$  In this paper, word, we mean a reduced word.  The length of the reduced word $w$ is defined by $\Sigma_{j=1}^s |a_j|.$  For any group $G$, let $G^{(n)}$ denote the direct product of $n$ copies of $G$. Corresponding to the word $w$, the word map $w$ on $G$ is an evaluation map from $G^{(n)}$ to $G$, defined as  $$(g_1,g_2, \ldots, g_n) \mapsto w(g_1,g_2, \ldots, g_n) = \prod_{j=1}^s g_{i_j}^{a_j}.$$
The image of the word map $w$ on $G$ is denoted by $G_w.$ It is clear that $G_w$ is nonempty as the identity element $e\in G_w.$
For example, if $w=x_1x_2x_1^{-1}x_2^{-1}\in F_2$ is a word of length $4$ and $G$ is an abelian group, then $G_w=\{e\}=G_{w^{-1}}.$

 Word maps have become an important tool to explore more about groups. One can refer~\cite{Bray}, \cite{cockeho}, \cite{Segal}, \cite{nikolov2007characterization}   to see the applications of word maps on groups. In particular, the image of a word map on a finite simple group is studied by~\cite{Jambor}, \cite{kassabov}, \cite{Levy}, \cite{ore}.

 It is easy to see that the image of a word map on a group $G$ corresponding to any given word is invariant under  endomorphism of $G$.  A group $G$ is said to satisfy the {\em property $\mathcal{P}$} if for every subset $S$ of $G$ with $e \in S$ and $S$ is invariant under every automorphism of $G$ can be seen as an image of a word map. In 2012, Alexander Lubotzky~\cite{Lubotzky} proved that all finite simple groups satisfy the property $\mathcal{P}.$  Carolyn Ashurst~\cite{Car}, in her thesis, asked that for any finite group $G$, whether under every word map, the cardinality of preimage of an element $g$ is same as the cardinality of preimage of $g^{-1}.$ Cocke and Ho~\cite{cocke} by using Lubotzky's result showed that this is not true and defined {\em chiral word}. A word $w \in F_n$ is said to be {\em chiral} if there exists a group $G$ such that $G_w \neq G_{w^{-1}}$. We define a word to be {\em achiral} if it is not {\em chiral}. A group $G$ is said to be {\em chiral} if there exists a word $w \in F_n$ for some $n$, such that $G_w \neq G_{w^{-1}}$. If $G$ is not chiral then we call $G$ as  {\em achiral}.  Let $X$ be a subset of a group $G$. Then $X$ is inverse closed if $X=X^{-1},$ where  $X^{-1} = \{ g^{-1} | g \in X \}.$ It is clear that ${G_w}^{-1} = G_{w^{-1}}$. Hence a word $w$ is achiral if for every group $G,$ $G_w = (G_w)^{-1}.$ That is a word $w$ is achiral if and only if $G_w$ is inverse closed for every group $G.$ For example, the sets $G^k=\{g^k|g\in G\},k\in \Z$ are inversely closed for any group $G,$ hence the words $w=x_1^k,k\in \Z$ are achiral. Consequently, every word from $F_1$ is achiral. Let $G$ be a group then  the subset $S=\{ghg^{-1}h^{-1}|g,h\in G\}$ of $G$ is inverse closed. Hence the word $w=x_1x_2x_1^{-1}x_2^{-1}$ is achiral. If the word map $w$ is onto for every group $G$ {\it i.e.,} $G_w=G,$ then $w$ is achiral. For example, $w=x_3x_1x_2x_1^{-1}x_2^{-1}\in F_3$ and $w=x_1x_2x_1x_2^2\in F_2$ are achiral. Let
 $$A_n=\{w\in F_n|\;\mbox{$w$ is achiral}\}$$
 $$B_n=\{w\in F_n|\; \mbox{for every group}\; G,\;G_w=G\}.$$ Then it is clear that $B_n\subseteq A_n.$
 It is known that the Nielsen transformations do not change the image of a word map in a group. Let $Aut(F_n)$ denote the set of all automorphisms of $F_n.$  Thus if $w\in F_n$ and $\sigma\in Aut(F_n),$ then $G_w=G_{\sigma(w)}.$ Hence the automorphic image of an achiral word is achiral. We will prove this result as a consequence of Proposition~\ref{pro:endo} in  section~\ref{sec:main}.  This observation will help us to provide a few elements in $B_n.$ A word $w \in F_n$ is said to be {\em primitive} if $w$ can be a part of a basis of $F_n$. A primitive word will always be achiral as it is an automorphic image of achiral word $x_1$. Since $x_1\in B_n,$ the set of all primitive words is a  subset of $B_n.$ But there are words in $B_n$ which are not primitive , for example $x_1^2x_2^3$. Akbar Rhemtulla~\cite{Rehmtulla} gave a complete description of $B_n$ which is as follows.

\begin{lemma}[Akbar Rhemtulla~\cite{Rehmtulla}]
Let $w \in F_n.$ Then $w\in B_n$ if and only if there exists integers $e_1,\ldots,e_n$ with $gcd(e_1,e_2,\ldots,e_n)=1$ such that $w\in {x_1}^{e_1}\ldots{x_n}^{e_n}{F_n}^{\prime},$ where ${F_n}^{\prime}$ is the commutator subgroup of $F_n.$
\end{lemma}
It is known \cite{Benkoski} that the probability of an element $(a_1,a_2,\ldots,a_k)\in \Z^n$ such that\\ $gcd(a_1,a_2,\ldots,a_k)=1$ is $\frac{1}{\zeta(k)}.$ Hence the probability for a word $w\in F_n$ belongs to $A_n$ is positive. In  section~\ref{sec:main}, we provide a  necessary and sufficient condition for a word $w\in F_n$ is achiral. By using this characterization, one can get more elements in $A_n.$ In the last subsection we will prove that it is enough to apply Whitehead's algorithm for Engel word's achirality.

\section{Main Results}\label{sec:main}
The following result was stated in the thesis
 of William Cocke~\cite{William}. We are giving the proof for the sake of completeness.
\begin{pro}\label{pro:endo}
Following two statements are equivalent for a word $w \in F_n$:
\begin{enumerate}
\item $w\in A_n.$
\item There exists an endomorphism $\psi$ of $F_n$ such that $\psi(w) = w^{-1}.$
\end{enumerate}
\end{pro}
\begin{proof}
Let $w\in F_n.$ We can view $w$  as a word map on $F_n$ as $w :{F_n}^n \longrightarrow F_n$ and   $w(x_1,\ldots,x_n) = w$ {\it i.e.,} $w\in (F_n)_w.$ Suppose $w$ is achiral. If we take $G=F_n$, then $w^{-1} \in{(F_n)}_{w^{-1}}= {(F_n)}_w.$ It implies that there exists $w_1, \ldots, w_n \in F_n$ such that $w(w_1, \ldots, w_n) = w^{-1}$. This suggests that the endomorphism $\psi$ of $F_n$ defined by the map $x_i \mapsto w_i$ $\forall i \in \{1,2,\ldots,n\}$ will give $\psi(w) = w^{-1}.$

Let $\psi$ be an  endomorphism  of $F_n$ such that $\psi(w) = w^{-1}.$ Let $G$ be a group and $g\in G_w.$
Then there exists $g_1,\ldots,g_n \in G$ such that $w(g_1, \ldots, g_n) = g.$ We have $$\psi(w(x_1,x_2,\ldots, x_n))=\psi(w)=w^{-1}.$$
Thus $g^{-1}=w^{-1}(g_1,\dots,g_n)=w((\psi(x_1),\ldots,\psi(x_n))(g_1,\ldots,g_n))\in G_w.$
\end{proof}
Let $End(F_n)$ denote the set of all endomorphisms of $F_n.$ Then $w\in F_n$ is achiral if and only if there exist $\sigma\in End(F_n)$ such that $\sigma(w)=w^{-1}.$
 The following observations are immediate from the Proposition~\ref{pro:endo}. 
\begin{cor}\label{cor:endo}
  Let $w\in F_n.$
  \begin{enumerate}
      \item \label{cor:endo:1} If $w$ is achiral, then $\sigma(w)$ will also be achiral for all $\sigma \in Aut(F_n).$
      \item \label{cor:endo:2} $w^n$ is achiral for all $n \in \Z.$
       \item \label{cor:endo:3} If $w^n$ for $n \in \Z\setminus\{0\}$ is achiral then so is $w$.
    \end{enumerate}
    \end{cor}
     \begin{description}
   \item[Proof of Part~\ref{cor:endo:1}.] Since $w$ is achiral, there exists an endomorphism $\phi$ taking $w$ to $w^{-1}$. Then the endomorphism defined by $\sigma \circ \phi \circ \sigma^{-1}$  takes $\sigma({w})$ to  $\sigma(w)^{-1}.$

    \item [Proof of Part~\ref{cor:endo:2}.] Let $f:F_n\to F_n$ be an endomorphism  such that $f(w) = w^{-1}$. Then $f(w^k) = {f(w)}^k = w^{-k} \;\;\forall k \in \Z.$
    \item[Proof of Part~\ref{cor:endo:3}.] Let $\phi$ be endomorphism inverting the word $w^n$. Suppose $\phi(w) = u$, then we have $\phi(w^n) = u^n = w^{-n}$. Hence $u = w^{-1}$.
    \end{description}

    We have seen that the automorphic image of an achiral word will always be achiral. Cocke and Ho~\cite{cocke} proved that a homomorphic image of achiral group will always be achiral group. But a homomorphic image of achiral word need not be achiral. In particular, let $w$ be a chiral word, then we can construct a required  homomorphism  $\phi$ such that  $\phi(x_1)=w.$ The next result states that a word being achiral is independent of free groups.
    \begin{pro}
        Let $w \in F_m \subseteq F_n$ for $m \leq n.$ If $w$ is achiral (chiral) in $F_m$ if and only if $w$ is achiral (chiral) in $F_n$.
    \end{pro}

    \begin{proof}
          We will prove for achiral, chiral part will automatically follow.  Let $x_1,\ldots,x_m$ be a basis of $F_m$ and $x_1,\ldots,x_n$, extending the basis of $F_m$, be the chosen basis of $F_n$. Let $i$ be the inclusion map $F_m \overset{i}{\hookrightarrow} F_n$ and $\alpha$ be the surjection map $F_n \overset{\alpha}{\twoheadrightarrow} F_m$ defined by the map $\alpha(x_i)$ goes to $x_i$ if $1 \leq i \leq m$, $1$ otherwise.
          
          Suppose $w$ is achiral in $F_m.$ Then there exists an endomorphism $\phi$ of $F_m$ such that $\phi(w) = w^{-1}.$ We can extend this endomorphism to an endomorphism $\Bar{\phi}$ of $F_n$ sending each basis element which is not in $F_m$ to identity. Thus we have $\Bar{\phi}(w)=w^{-1}.$ 

          Let $w$ be achiral in $F_n.$ Then there exists an endomorphism $\psi$ of $F_n$ such that $\psi(w) = w^{-1}.$
          $$F_m {\overset{i} \hookrightarrow} F_n {\overset{\psi}{\rightarrow} F_n \overset{\alpha}{\twoheadrightarrow}} F_m$$
          The composition of these maps $\alpha \circ \psi \circ i $ is an endomorphism of $F_m$ such that $\alpha \circ \psi \circ i (w) = w^{-1}.$
    \end{proof}
    For a shorter proof one can use chiral words.
    \begin{rem}
    In fact the proof of above proposition also suggests that if there exists an automorphism of $F_m$ inverting the word $w$ if and only if there exists an automorphism of $F_n$ inverting the word $w.$
    \end{rem}
 
\begin{cor} \label{cor:endo1} There exists an automorphism inverting the following words.
    
    \begin{enumerate} 
   
      \item \label{cor:endo:4}  Every palindromic word is achiral.
    \item \label{cor:endo:5}  Every word of type $x_1^{m_1}x_2^{m_2}$, where $m_1 , m_2 \in \Z$, is achiral.
    \item \label{cor:endo:6}   Every word of type $x_1^mx_2^{\epsilon_1}x_1^nx_2^{\epsilon_2}$ where $m,n,\epsilon_i \in \Z$ and $\epsilon_1 = \pm \epsilon_2$ is achiral.
    
      \end{enumerate}
 \end{cor}
 \begin{description}
     \item [Proof of Part~\ref{cor:endo:4}.] Let $w \in F_n$ be a palindromic word. Take the automorphism generated by $f(x_i) = {x_i}^{-1}$ for $i \in \{1,2, \ldots,n \}.$ Then $f(w) = w^{-1}.$
     
    \item [Proof of Part~\ref{cor:endo:5}.] Here we prove the existence of automorphism inverting the word for $F_2$. Hence by remark of previous proposition the result holds. It is enough to prove for $x_1^{m_1}x_2^{m_2}$. Take automorphism $f$ generated by sending $x_1$ to $x_2^{-m_2}x_1^{-1}x_2^{m_2}$ and $x_2$ to $x_2^{-1}$. Then $f(x_1^{m_1}x_2^{m_2}) = x_2^{-m_2}x_1^{-m_1}.$

    \item [Proof of Part~\ref{cor:endo:6}.] Here we prove the existence of automorphism inverting the word for $F_2$. Hence by remark of previous proposition the result holds. When $\epsilon_1 = \epsilon_2$, take automorphism generated by sending $x_1$ to $x_1^{-1}$ and $x_2$ to $x_1^mx_2^{-1}x_1^{-m}$. If $\epsilon_1 = -\epsilon_2$,  take automorphism generated by sending $x_1$ to $x_1^{-1}$ and $x_2$ to $x_1^mx_2x_1^{-m}.$
    
  \end{description}

Makanin~\cite{Makanin}  has proved that whether a system of equations over a free group has a solution or not is algorithmically decidable. The following result shows that the achirality of a word is a decidable problem.

\begin{lemma}
   Showing achirality of a word $w$ in a free group $F_n$ is equivalent to solving an equation over $F_n.$
  \end{lemma}

  \begin{proof}
  Let $w = \prod_{j=1}^l x_{i_j}^{t_j}\in F_n$ be achiral.  From the Proposition~\ref{pro:endo} there exist an endomorphism $\phi$ of $F_n$ such that $\phi(w)=w^{-1}.$  Since endomorphism of a free group is determined by the images of its generators, let $\phi(x_i) = w_i$  So $\phi(w) = \prod_{j=1}^l w_{i_j}^{t_j}$. Thus $w$ is achiral if and only if
   $$\prod_{j=1}^l x_{i_j}^{t_j}\prod_{j=1}^l w_{i_j}^{t_j} = 1.$$
  Hence checking a given word $w\in F_n$ is achiral is equivalent to solving an equation in $F_n.$
 \end{proof}

\subsection{Achiral words in $F_2$}
Let $F_2$ be a free group of rank $2$ with alphabet $X=\{x_1,x_2\}.$ If $w\in F_2,$ then $w$ is reduced word with letters $x_1,x_2,x_1^{-1},x_2^{-1}.$ We denote the length of word $w$ by
$\ell(w).$ Let $\W_n$ denote the set of all words of length $n$ {\it i.e.,}
$$\W_n=\{w\in F_2| \ell(w)=n\}.$$
It is easy to see that
$|\W_n|=4\cdot 3^{n-1}.$ For example,
$$\W_1=\{x_1,x_2,x_1^{-1},x_2^{-1}\}.$$
$$\W_2=\{x_1^2,x_2^2,x_1^{-2},x_2^{-2}, x_1x_2,x_1^{-1}x_2^{-1}, x_2x_1, x_2^{-1}x_1^{-1},  x_1^{-1}x_2,x_1x_2^{-1},x_2^{-1}x_1,x_2x_1^{-1}\}.$$

Let $$\W^n=\{w\in \W_n|\exists \sigma\in\; Aut(F_2),\sigma(w)=w^{-1}\}.$$ Hence $\W^n\subseteq A_2.$
Before we proceed, given an automorphism $f\in \; Aut(F_2)$, we define $$\W_f^n=\{w\in \W^n|f(w)=w^{-1}\}.$$

If $f(x_1)=x_1^{-1}, f(x_2)=x_2^{-1}; g(x_1)=x_2^{-1},g(x_2)=x_1^{-1}$ and $h(x_1)=x_2, h(x_2)=x_1,$ then it is easy that
 $$\W_f^2=\{x_1^2,x_2^2,x_1^{-2},x_2^{-2}\}, \W_g^2=
 \{x_1x_2,x_1^{-1}x_2^{-1}, x_2x_1, x_2^{-1}x_1^{-1}\}, \W_h^2=\{x_1^{-1}x_2,x_1x_2^{-1},x_2^{-1}x_1,x_2x_1^{-1}\}.$$
 Thus $\W_2=\W^2.$ That is every word of length $2$ over $F_2$ is achiral.  From the observations listed in Corollary~\ref{cor:endo} and \ref{cor:endo1}, it is easy to see that
$\W_i=\W^i$ for $i\in \{1,2,3,4,5\}$ and $W_6=A_6$ but the word $x_1^2x_2^2x_1x_2^{-1}\in \W_6\setminus \W^6$ and it is an  achiral word of shortest length such that it is not inverted by any automorphism (see~\cite{Gardam}, Remark 3.7). In short, Our result above with Rhemtulla's result are enough to prove that all words of length $6$ are achiral.

 \subsection{Test word and achiral word}
 Let $w\in F_n.$ We define $$End_w=\{f\in End(F_n)|f(w)=w^{-1}\}\;\;\;Aut_w=\{f\in Aut(F_n)|f(w)=w^{-1}\}.$$
It is clear that $Aut_w\subseteq End_w.$ From Proposition~\ref{pro:endo} we have $w\in A_n$ if and only if   $End_w\ne \emptyset.$ Let $w\in F_n$ be a primitive word. Then we can construct an endomorphism $\sigma$ satisfying $\sigma \in End_w\setminus Aut_w.$
As discussed earlier if $w=x^2y^2xy^{-1},$ then $w\in A_2$ that is $End_w\ne \emptyset$ but  $Aut_w=\emptyset.$ Let $D_n=\{w\in A_n|End_w=Aut_w\}.$
\begin{definition}
 A word $w\in F_n$ is a {\em test word} if every endomorphism which fixes $w$ is an automorphism.
\end{definition}

We denote set of all test words of $F_n$ as $T_n$ that is $$T_n=\{w\in F_n|\; \mbox{if  $\sigma\in End(F_n)$ with $\sigma(w)=w,$ then $\sigma\in Aut(F_n)$}\}.$$ Turner~\cite{Turner} gave a criterion to check a given word is a test word or not.

\begin{definition}
A subgroup $H$ of $G$ is said to be retract if there exists an endomorphism $\phi$ such that $\phi: G \longrightarrow H$ with the property $\phi(h) = h$ for all $h \in H.$
\end{definition}

\begin{lemma}[Turner~\cite{Turner}]
A word $w \in F_n$ is a test word if and only if it is not in any proper retract of $F_n$.
\end{lemma}

 Turner also showed that $T_2={F_2}^{\prime}\setminus\{e\}.$

\begin{theorem}\label{thm:Tn}
If $n\in \N,$ then $T_n \cap A_n = D_n.$
\end{theorem}

\begin{proof}
Let $w \in (T_n \cap A_n) \setminus D_n.$ Then there exists $\phi\in End_w\setminus Aut_w.$ Which leads to a contradiction that $\phi^2$ fixes $w$ hence $\phi\in Aut_w.$ Hence $(T_n \cap A_n) \setminus D_n=\emptyset$ or equivalently $D_n\subseteq T_n \cap A_n.$ Let $w\in T_n \cap A_n$ and $\phi\in End_w.$ Then $\phi^2$ fixes $w.$ Hence  $\phi^2\in Aut(F_n)$ so is $\phi.$ Thus $\phi\in Aut_w.$
\end{proof}
The above theorem suggests that all words in $T_n \setminus D_n$ are chiral. Cocke and Ho has given an explicit example of chiral word  $[x^{ 440}(x^{ 440})^{(y^{ 440} )}x^{  440}, (y^{  440})^{(x^{ 440} y^{  440} )}y^{ 440}]$  which lies inside $T_2\setminus D_2$ in $F_2$. They asked whether engel words $e_n = [x,_ny]$ are achiral or not. From  Theorem~\ref{thm:Tn}  it is enough to check for an automorphism inverting $e_n$. For this case, we have Whitehead's algorithm~\cite{whitehead}.

It is easy to see that $B_1 =\{x,x^{-1}\}$ and $D_1 = T_1=F_1\setminus \{e\}.$ But if $n \geq 2,$ then we have  $B_n \cap T_n = \emptyset.$ As let $w \in B_n.$ Take $H = \langle w \rangle.$ We will show that $H$ is a proper retract of $F_n$. Hence $w$ can't be a test word.  Write $w = x_1^{r_1}\ldots x_n^{r_n}c$ where $c \in {F_n}^{\prime}$. Since $w \in B_n,$  we have $gcd(r_1, \ldots, r_n) = 1.$ Hence there exists $m_1, \ldots, m_n$ such that $r_1m_1 + \ldots + r_nm_n = 1.$ Define a homomorphism $\phi : F_n \longrightarrow H$ generated by $\phi(x_i) = w^{m_i}$ for all $i \in \{1, \ldots, n\}.$ So we get $\phi(w) = w.$ Hence $H$ is a proper retract of $F_n.$ So by Turner's result, $w$ can not be test word.
\bibliographystyle{siam}
\scriptsize
\bibliography{mybib}
\end{document}